\documentclass[reqno,a4paper,12pt]{amsart} 

\usepackage{amsmath,amscd,amsfonts,amssymb}
\usepackage{mathrsfs,dsfont}

\usepackage{tikz}
\usetikzlibrary{patterns}
\usepackage{float}
\usepackage{appendix}
\usepackage{caption}
\usepackage{subcaption}
\usepackage{enumerate}

\numberwithin{equation}{section}
\numberwithin{figure}{section}

\def\R{\mathbb{R}}

\def\Z{\mathbb{Z}}

\def\1{\mathds{1}}

\def\dH{\dim_{\mathcal{H}}}

\renewcommand\leq{\leqslant}
\renewcommand\geq{\geqslant}

\renewcommand\hat{\widehat}

\theoremstyle{plain}
\newtheorem{thm}{Theorem}[section]
\newtheorem{lem}[thm]{Lemma}

\newtheorem{prop}[thm]{Proposition}

\newtheorem*{claim*}{Claim}
\newtheorem*{thm*}{Theorem}

\theoremstyle{definition}

\newtheorem*{definition*}{Definition}
\newtheorem*{remarks*}{Remarks}
\newtheorem*{remark*}{Remark}
\newtheorem{remark}[thm]{Remark}

\newenvironment{enumerate-math}
{\begin{enumerate}
\addtolength{\itemsep}{5pt}
}
{\end{enumerate}}

\newenvironment{enumerate-text}
{\begin{enumerate}
\addtolength{\itemsep}{5pt}
}
{\end{enumerate}}

\begin{document}

\title{Hausdorff dimension of pinned distance sets and the $L^2$-method}

\author{Bochen Liu}
\address{Department of Mathematics, the Chinese University of Hong Kong, Shatin, N.T., Hong Kong}
\email{Bochen.Liu1989@gmail.com}

\thanks{The author is supported by the grant CUHK24300915 from the Hong Kong Research Grant Council}
%\subjclass[2010]{05B45, 52C20, 51M20}
\date{}

%\keywords{multiple tiling, periodic tiling}

\begin{abstract}
	 We prove that for any compact set $E\subset \mathbb{R}^2$, $\dim_{\mathcal{H}}(E)>1$, there exists $x\in E$ such that the Hausdorff dimension of the pinned distance set
	 $$\Delta_x(E)=\{|x-y|: y \in E\}$$
	 is no less than $\min\left\{\frac{4}{3}\dim_{\mathcal{H}}(E)-\frac{2}{3}, 1\right\}$. This answers a question recently raised by Guth, Iosevich, Ou and Wang, as well as improves results of Keleti and Shmerkin.

%	 We also show that in $\R^d$, $d\geq 2$, the Mattila integral and spherical averaging estimates imply the same (Hausdorff) dimensional lower bound on both distance sets and pinned distance sets, although in the realm of Lebesgue measure it seems only working on distance sets.

	 %We also discuss a bit about relation between distance problem and Kakeya problem, which might lead to further work.
\end{abstract}
\maketitle

\section{Introduction}
\subsection{Falconer distance conjecture and pinned distance problem}
Falconer distance conjecture \cite{Fal85} is one of the most famous open problems in geometric measure theory, which states that for any compact set $E\subset\R^d$, $d\geq 2$, $\dH(E)>\frac{d}{2}$, its distance set
$$\Delta(E)=\{|x-y|:x,y\in E\}$$
has positive Lebesgue measure. 

Throughout this paper we use $\dH$ to denote Hausdorff dimension. Also dimension refers to Hausdorff dimension unless stated otherwise. 

A stronger version of Falconer distance conjecture is the pinned distance problem, which asks whether there exists $x\in E$ such that the pinned distance set
$$\Delta_x(E)=\{|x-y|:y\in E\} $$
has positive Lebesgue measure.
\subsection{The $L^2$-method}
One direction to study these problems is to investigate how large $\dH(E)$ needs to be to ensure that $\Delta(E)$, $\Delta_x(E)$ have positive Lebesgue measure. In this paper we focus on the pinned version. In fact the best currently known dimensional exponents on distances and pinned distances match.

Given a probability measure $\mu_E$ on $E$, one can define a natural measure $\nu_x$ on $\Delta_x(E)$ by
$$\int f(t)\,d\nu_x(t)= \int f(|x-y|)\,d\mu_E(y).$$
Equivalently, $\nu_x=d^x_*(\mu_E)$, where $d^x(y)=|x-y|$. 

To show the support of $\nu_x$ has positive Lebesgue measure, it suffices to show the Radon-Nikodym derivative $\frac{d\,\nu_x}{d\,t}\in L^p$ for some $p>1$. When $p=2$, the author \cite{Liu18} discovered the following identity,
\begin{equation}
	\label{L2-id}
	\int_0^\infty |f*\omega_t(x)|^2\,t^{d-1}\,dt = \int_0^\infty |f*\widehat{\omega_r}(x)|^2\,r^{d-1}\,dr,
\end{equation}
for any Schwartz function $f$ on $\R^d$ and any $x\in\R^d$. Here $\omega_r$ denotes the normalized surface measure on $r S^{d-1}$. It implies that, to show $\frac{d\,\nu_x}{d\,t}\in L^2(t^{d-1}\,dt)$ for $\mu_E$-a.e. $x\in E$, it suffices to show
\begin{equation}
	\label{L2-id-mu}
	\iint |\mu_E*\widehat{\omega_r}(x)|^2\,r^{d-1}\,dr\,d\mu_E(x)
\end{equation}
is finite, which is closely related to Fourier restriction in harmonic analysis. 

With the help of this $L^2$-method, the best currently known dimensional threshold to ensure $|\Delta_x(E)|>0$ for some $x\in E$, as well as the best to ensure $|\Delta(E)|>0$, is
\begin{equation}
		\label{best-known}
		\dH(E)>\begin{cases}\frac{5}{4},&d=2\ (\text{Guth, Iosevich, Ou, Wang, \cite{GIOW18}, 2018})\\1.8,&d=3\ (\text{Du, Guth, Ou, Wang, Wilson, Zhang, \cite{DGOWWZ18}, 2018})\\\frac{d}{2}+\frac{1}{4}+\frac{1}{8d-4},&d\geq 4\ (\text{Du, Zhang, \cite{DZ18}, 2018})
		\end{cases}.
	\end{equation}

As a remark, Guth-Iosevich-Ou-Wang's argument in the plane is a variant of the $L^2$-method. They first decompose $\mu_E=\mu_{E, good}+\mu_{E, bad}$, then show $\mu_{E, bad}$ is negligible and $\nu_{x, good}:=d^x_*(\mu_{E, good})$ is in $L^2$. They also gave examples to show that if one only works on the $L^2$-norm of $\nu_x$, then no result better than $\dH(E)>\frac{4}{3}$ could be obtained. We remind the reader that $\dH(E)>\frac{4}{3}$ is the previous record in the plane, followed by \eqref{L2-id-mu} and a spherical averaging estimate of Wolff \cite{Wol99}. We will discuss more about Guth-Iosevich-Ou-Wang's argument in Section \ref{prelim}.

\subsection{Dimension of (pinned) distance sets}
Another direction to study (pinned) distance problem is, given $E\subset\R^d$, $\dH(E)>\frac{d}{2}$, one can investigate the dimension of $\Delta_x(E)$, $\Delta(E)$. There is a natural way to apply the $L^2$-method to dimension of (pinned) distance sets. To show $\dH(\Delta_x(E))\geq \tau$ (similarly for $\dH(\Delta(E))$), it suffices to show the $\tau$-energy integral of $\nu_x$,
$$I_\tau(\nu_x) := \iint |t-t'|^{-\tau}\,d\nu_x(t)\,d\nu_x(t')=C\int|\hat{\nu_x}(\xi)|^2\,|\xi|^{-1+\tau}\,d\xi= C\,||\nu_x||^2_{L^2_{\frac{-1+\tau}{2}}}$$
is finite (see, for example, \cite{Mat15}, Theorem 3.10 for the expression of the energy integral using Fourier transform). If one studies the $L^2$-norm of $\nu_x$ via \eqref{L2-id-mu} and harmonic analysis, then arguments also work on the $L^2_{\frac{-1+\tau}{2}}$-norm of $\nu_x$. In dimension $3$ and higher, this method is still the best. In the plane, better results follow from investigating coverings and local structure of sets in different scales. The first result is due to Bourgain \cite{Bou03}, who found an absolute $\epsilon_0>0$ such that $\dH(\Delta(E))\geq\frac{1}{2}+\epsilon_0$ whenever $\dH(E)\geq 1$. The best currently known results are due to Keleti and Shmerkin \cite{KS18}, who proved
\begin{itemize} 
	\item given $E\subset\R^2$, $\dH(E)\in(1,\frac{4}{3})$, then
	\begin{equation}\label{KS-distance}\dH(\Delta(E))\geq\dH(E)\frac{147-170\dH(E)+60\dH(E)^2}{18(12-14\dH(E)+5\dH(E)^2)}\geq \frac{37}{54}=0.685\cdots;\end{equation}
	\item given $E\subset\R^2$, $\dH(E)>1$, then there exists $x\in E$ such that\begin{equation}\label{KS-pinned}\dH(\Delta_x(E))\geq\min\left(\frac{2}{3}\dH(E), 1\right).\end{equation}
\end{itemize}

\subsection{A question raised by Guth, Iosevich, Ou and Wang}
As we remarked right after \eqref{best-known}, authors in \cite{GIOW18} decompose $\mu_E=\mu_{E, good}+\mu_{E, bad}$ and consider the $L^2$-norm of $\nu_{x, good}:=d^x_*(\mu_{E,good})$. It is pointed out in the Appendix of \cite{GIOW18} that neither $\nu_{x, good}$ is supported on $\Delta_x(E)$, nor $\nu_{x, bad}$ is negligible on energy integrals. Therefore, although good estimates on $I_\tau(\nu_{x, good})$ still follow naturally, it does not imply any result on $\dH(\Delta_x(E))$. If it had worked, it would follow that for any compact $E\subset\R^2$, $\dH(E)>1$, we have
\begin{equation}\label{goal}\dH(\Delta_x(E))\geq \min\left\{\frac{4}{3}\dH(E)-\frac{2}{3},1\right\} \end{equation}
for some $x\in E$, which improves \eqref{KS-pinned} when $\dH(E)>1$, and in particular improves \eqref{KS-distance} when $\dH(E)>1.028\cdots$.

Therefore it is reasonable to expect \eqref{goal} to hold. In this paper we give a positive answer to this expectation.
\begin{thm}\label{Dimension-Pin}
	Given any compact set $E\subset\R^2$, $\dH(E)>1$ and $\tau\in (0 ,1)$, then
	$$\dH\left\{x\in \R^2: \dH(\Delta_x(E))<\tau \right\}\leq \max\left\{2+3\tau-3\dH(E), 2-\dH(E)\right\}. $$
	In particular, for any compact set $E\subset\R^2$, $\dH(E)>1$, there exists $x\in E$ such that
	$$\dH(\Delta_x(E))\geq \min\left\{\frac{4}{3}\dH(E)-\frac{2}{3}, 1\right\}. $$
\end{thm}

\begin{remark}
	Shortly after this paper was made public, Shmerkin \cite{Shm18} plug Guth-Iosevich-Ou-Wang's estimate \cite{GIOW18} into Keleti-Shmerkin's framework \cite{KS18} and obtained
	$$\dH(\Delta(E))\geq \frac{40}{57}=0.702\cdots;$$ $$\dH(\Delta_x(E))>\frac{29}{42}=0690\cdots, \text{for some }x\in E, $$
	given $E\subset\R^2$, $\dH(E)>1$. This is better than Theorem \ref{Dimension-Pin} when $\dH(E)$ is very close to $1$.
\end{remark}

{\bf Notation.} $X\lesssim Y$ means $X\leq CY$ for some constant $C>0$. $X\approx Y$ means $X\lesssim Y$ and $Y\lesssim X$. $X\lesssim_\epsilon Y$ means $X\leq C_\epsilon Y$ for some constant $C_\epsilon>0$, depending on $\epsilon$.

%For any set $A\subset\R^d$, $|A|$ denotes its $d$-dimensional Lebesgue measure.

Let $\phi\in\ C_0^\infty(\R^d)$, $\phi\geq 0$, $\int \phi =1$, and $\phi\geq 1$ on $B(0,\frac{1}{2})$. Denote $\phi_\delta(\cdot)=\frac{1}{\delta^d}\phi(\frac{\cdot}{\delta})$ and $\mu^{\delta}=\mu*\phi_\delta$ for any Radon measure $\mu$ on $\R^d$.

$RapDec(R)$ means for any $N>0$, there exists a constant $C_N>0$ such that $RapDec(R)\leq C_N R^{-N}$.

Denote $d\omega_r$ as the normalized surface measure on $rS^{d-1}$. Also denote $d\omega=d\omega_1$.

$\widehat{f}(\xi):=\int e^{-2\pi i x\cdot\xi} f(x)\,dx$ denotes the Fourier transform.

Denote $d^x(y)=|x-y|$, the Euclidean distance between $x$ and $y$.

\section{Review of Guth-Iosevich-Ou-Wang's argument}\label{prelim}
Let $E\subset\R^d$ be a compact set. It is well known that for any $s_E<\dH(E)$, there exists a probability measure $\mu_E$ on $E$, called a Frostman measure, such that
\begin{equation}\label{Frostman}\mu_E(B(x,r))\lesssim r^{s_E},\ \forall\ x\in\R^d,\ \forall \ r>0. \end{equation}
For any $s<s_E$, the $s$-energy integral of $\mu_E$,
\begin{equation}\label{energy-integral}I_{s}(\mu_E):=\iint |x-y|^{-s}\,d\mu_E(x)\,d\mu_E(y)= C \int |\widehat{\mu_E}(\xi)|^2\,|\xi|^{-d+s}\,d\xi,\end{equation}
is finite. For more details, see, for example, \cite{Mat15}, Section 2.5.

Suppose $E_1, E_2\subset\R^2$ are compact, $dist(E_1, E_2)>0$, and $\mu_1$, $\mu_2$ are Frostman measures on $E_1$, $E_2$ satisfying \eqref{Frostman}, with $s_1=s_2=\alpha$. In \cite{GIOW18}, it is proved that there is a decomposition $\mu_1=\mu_{1, good}+\mu_{1, bad}$, where $\mu_{1, good}$, $\mu_{1, bad}$ are complex-valued distributions, such that

\begin{itemize} \item (Proposition 2.1, \cite{GIOW18}) if $\alpha>1$, then there exists $E_2'\subset E_2$, $\mu_2(E_2')\geq 1-\frac{1}{1000}$, such that for each $x\in E_2'$,
	\begin{equation}\label{prop-2.1-GIOW}\int |d^x_*(\mu_1)(t) - d^x_*(\mu_{1, good})(t) |\,dt<\frac{1}{1000};\end{equation}
	\item (Proposition 2.2, \cite{GIOW18}) if $\alpha>\frac{5}{4}$, then \begin{equation}\label{prop-2.2-GIOW}\iint |d^x_*(\mu_{1, good})(t)|^2\,dt\,d\mu_2(x)<\infty.\end{equation}
\end{itemize}
As a consequence, $\Delta_x(E)$ has positive Lebesgue measure for some $x\in E$ whenever $\dH(E)>\frac{5}{4}$ (see \cite{GIOW18}, Section $2$). 

In fact what is proved in \cite{GIOW18} is the following quantitative version. Throughout this paper $RapDec(R)$ means for any $N>0$, there exists a constant $C_N>0$ such that $RapDec(R)\leq C_N R^{-N}$.
\begin{prop}\label{what-proved}
With notation above, there exists a constant $c=c(\alpha)>0$ such that for any $R_0\gg 1$ and $1\gg\delta>0$, one can decompose $\mu_1=\mu_{1, good}+\mu_{1, bad}$, where $\mu_{1, good}$, $\mu_{1, bad}$ are complex-valued distributions, such that
\begin{itemize} \item if $\alpha>1$, then there exists $E_2'\subset E_2$, $\mu_2(E_2')\geq 1-R_0^{-c\delta}$, such that for each $x\in E_2'$,
	\begin{equation}\label{prop-2.1-proved}\int |d^x_*(\mu_1)(t) - d^x_*(\mu_{1, good})(t) |\,dt\lesssim R_0^{-c\delta};\end{equation}
	\item there exists a constant $C(\delta)>0$ such that \begin{equation}\label{prop-2.2-proved}\iint |d^x_*(\mu_{1, good})(t)|^2\,dt\,d\mu_2(x)\lesssim C(\delta)R_0^{O(1)}\int |\widehat{\mu_1}(\xi)|^2\,|\xi|^{-\frac{\alpha+1}{3}+ O(\delta)}\,d\xi+RapDec(R_0).\end{equation}
\end{itemize}
Here all implicit constants may depend on $dist(E_1, E_2)$, $\alpha$, and the implicit constant in \eqref{Frostman}.
\end{prop}

To see why \eqref{prop-2.1-proved} holds, we refer to Section 3 in \cite{GIOW18} (see the last line in {\it Proof of Proposition 2.1 using Lemma 3.6}). In fact authors in \cite{GIOW18} proved \eqref{prop-2.1-proved} first and then choose a large $R_0$ to obtain \eqref{prop-2.1-GIOW}. 

To see why \eqref{prop-2.2-proved} holds, we refer to Section 5 in \cite{GIOW18} (see (5.20) and Proposition 5.3). We point out that in their Proposition 5.3 $\epsilon=O(\delta)$ (see the equation before (5.22) in \cite{GIOW18}), and $C(R_0)=R_0^{O(1)}$. The error term $RapDec(R_0)$ appears when applying the identity \eqref{L2-id} with $\mu_{1, good}$ in place of $\mu_1$ (see the discussion before Lemma 5.2 in \cite{GIOW18}). In \cite{GIOW18} $R_0$ is a fixed large constant, while we shall choose different $R_0$ in different scales so rapid decay is very important.

In this paper we need the following version of Proposition \ref{what-proved}. Although it is not written explicitly in \cite{GIOW18}, it follows from the proof. Denote $\mu_E^{2^{-k}}=\mu_E*\phi_{2^{-k}}$, where $\phi\in\ C_0^\infty(\R^2)$, $\phi\geq 0$, $\int \phi =1$, $\phi\geq 1$ on $B(0,\frac{1}{2})$, and $\phi_{2^{-k}}(\cdot)=2^{2k}\phi(2^k\cdot)$.

\begin{prop}\label{what-needed}
Suppose $E, F$ are compact sets in the plane, $\dH(E)>1$, $\dH(E)+\dH(F)>2$, $dist(E, F)\gtrsim 1$, and $\mu_E$, $\mu_F$ are Frostman measures on $E, F$ satisfying \eqref{Frostman}. Then there exists a constant $c=c(s_E, s_F)>0$ such that for any $1\gg\delta>0$, $2^k\gg R_0\gg 1$, one can decompose $\mu^{2^{-k}}_E=\mu^{2^{-k}}_{E, good}+\mu^{2^{-k}}_{E, bad}$ , where $\mu^{2^{-k}}_{E, good}$, $\mu^{2^{-k}}_{E, bad}$ are complex-valued Schwartz functions, 
 such that
\begin{itemize} \item if $1<s_E<\dH(E)$, $2-s_E<s_F<\dH(F)$, then there exists $F'\subset F$, $\mu_F(F')\geq 1-R_0^{-c\delta}$, and for each $x\in F'$,
	\begin{equation}\label{prop-2.1}\int |d^x_*(\mu^{2^{-k}}_E)(t) - d^x_*(\mu^{2^{-k}}_{E, good})(t) |\,dt\lesssim R_0^{-c\delta};\end{equation}
	\item there exists a constant $C(\delta)>0$ such that \begin{equation}\label{prop-2.2}\iint |d^x_*(\mu^{2^{-k}}_{E, good})(t)|^2\,dt\,d\mu_F(x)\lesssim C(\delta) R_0^{O(1)}\int |\widehat{\mu^{2^{-k}}_E}(\xi)|^2\,|\xi|^{-\frac{s_F+1}{3}+ O(\delta)}d\xi+RapDec(R_0).\end{equation}
\end{itemize}
Here all implicit constants may depend on $dist(E, F)$, $s_E, s_F$, and the implicit constant in \eqref{Frostman}, but are independent in $k$.
\end{prop}

Now we explain why Proposition \ref{what-needed} holds. 

In the proof of \eqref{prop-2.1-proved}, one needs $\alpha>1$ only when applying Orponen's radial projection theorem \cite{Orp19}, which still holds given $s_E>1$, $s_E+s_F>2$. 

Now it remains to explain why $\mu_E$ can be replaced by $\mu_E^{2^{-k}}$, $2^k\gg R_0\gg 1$. For \eqref{prop-2.1}, recall in the proof of \eqref{prop-2.1-proved} in \cite{GIOW18} (see the last line of the proof of Lemma 3.6 in \cite{GIOW18}), the implicit constant in \eqref{prop-2.1-proved} comes from the implicit constant in the following radial projection estimate due to Orponen. Denote $\pi_x:\R^2\backslash\{x\}\rightarrow S^1$ as the radial projection
$$\pi_x(y)=\frac{y-x}{|y-x|}. $$
Orponen (see \cite{Orp19}, (3.6)) proved that, if $s_E>1$, $s_E+s_F>2$, then when $p>1$ is small enough,
\begin{equation}\label{Orponen}\int ||(\pi_x)_*\mu_E||_{L^p(S^1)}^p\,d\mu_F(x)\lesssim_{\epsilon, p, s_E, s_F} I_{s_F-\epsilon}(\mu_F)^\frac{1}{2p} I_{s_E-\epsilon}(\mu_E)^\frac{1}{2}<\infty. \end{equation}
When we replace $\mu_E$ by $\mu_E^{2^{-k}}$,
$$I_{s}(\mu^{2^{-k}}_E)=C\int |\widehat{\mu_E}(\xi)|^2\,|\hat{\phi}(2^{-k}\xi)|^2\,|\xi|^{-d+s}\,d\xi\leq C\int |\widehat{\mu_E}(\xi)|^2\,|\xi|^{-d+s}\,d\xi = I_s(\mu_E),$$
so the upper bound of \eqref{Orponen} still holds uniformly in $k$. Hence \eqref{prop-2.1} holds uniformly in $k$.

For \eqref{prop-2.2}, we refer to Section 5 in \cite{GIOW18}. One can see that, the Frostman condition \eqref{Frostman} on $\mu_1$ is not used in the proof of \eqref{prop-2.2-proved}. It comes in only for the finiteness of \eqref{prop-2.2-proved} to obtain \eqref{prop-2.2-GIOW}. Therefore \eqref{prop-2.2-proved} remains valid when $\mu_1$ is replaced by any compactly supported finite measure whose support does not intersect $E_2$. Hence \eqref{prop-2.2} follows.

\section{Dimension of pinned distance sets}
We shall use the following criteria to determine dimension of pinned distance sets. This idea is now standard in geometric measure theory.
%In particular in combinatorial approaches to distance problems (see, e.g. \cite{shm17}, Lemma 6.2). However it was never combined with harmonic analysis and the $L^2$-method.
\begin{lem}\label{strategy}
Given a compact set $E\subset\R^d$, $x\in\R^d$ and a probability measure $\mu_E$ on $E$. Suppose there exist $\tau\in (0,1]$, $K\in\Z_+$, $\beta>0$ such that
	$$\mu_E(\{y: |y-x|\in D_k\})<2^{-k\beta}$$ 
	for any
	$$D_k=\bigcup_{j=1}^M I_j,$$
	where $k>K$, $M\leq 2^{k\tau}$ are arbitrary integers and each $I_j$ is an arbitrary interval of length $\approx 2^{-k}$. Then
	$$\dH(\Delta_x(E))\geq\tau. $$
\end{lem}

We give the proof for completeness.
\begin{proof}
If
$$\dH(\Delta_x(E))<\tau, $$
there exists $s\in(\dH(\Delta_x(E)), \tau)$ such that $\mathcal{H}^{s}(\Delta_x(E))=0$. By the definition of Hausdorff measure there exists an integer $N_0>0$ such that for any integer $N>N_0$, we can find a cover $\mathcal{I}$ of $\Delta_x(E)$ which consists of finitely many open intervals of length $<2^{-N}$, such that
$$\sum_{I\in\mathcal{I}} |I|^s\leq 1. $$

Denote $$\mathcal{I}_k=\{I\in\mathcal{I}: 2^{-k-1}\leq |I|< 2^{-k}\},\  k=N, N+1,\dots$$ and $$D_k=\bigcup_{I\in \mathcal{I}_k} I.$$
We may assume $N>\frac{s}{\tau-s}$. Then
$$\#(\mathcal{I}_k)\leq 2^{(k+1)s}\leq 2^{k\tau}.$$

Since 
$$\Delta_x(E)\subset \bigcup_{k\geq N} D_k$$ and $$\mu_E\left(\left\{y\in E: |x-y|\in \Delta_x(E)\right\}\right)=\mu_E(E)=1,$$ there exists $k_0\geq 0$ such that
$$\mu_E(\{y\in E: |y-x|\in D_{k_0+N}\})>\frac{1}{100 (k_0+1)^2}. $$

On the other hand, when $N>K$ the assumption in Lemma \ref{strategy} implies
$$\mu_E(\{y\in E: |y-x|\in D_{k_0+N}\})<2^{-(N+k_0)\beta},$$
which is a contradiction when $N$ is large enough so that $2^{-N\beta}<\inf_{k\geq 0}\frac{2^{k\beta}}{100(k+1)^2}$.
\end{proof}

\section{Proof of Theorem \ref{Dimension-Pin}}
For any compact $F\subset\R^2$ with $$\dH(F)>\max\left\{2+3\tau-3\dH(E),\ 2-\dH(E)\right\},$$ 
there exist Frostman measures $\mu_E$, $\mu_F$ on $E, F$ satisfying \eqref{Frostman}, with $1<s_E<\dH(E)$, $\max\{2+3\tau-3 s_E,\ 2-s_E\}<s_F<\dH(F)$. 

We can always find compact subsets $E'\subset E$, $F'\subset F$ such that $dist(E', F')\gtrsim 1$ and $\mu_E(E'), \mu_F(F')>0$. Therefore we may assume $dist (E, F)\gtrsim 1$. 

To prove Theorem \ref{Dimension-Pin}, it suffices to show for $\mu_F$-a.e. $x\in F$, we have
$$\dH(\Delta_x(E))\geq \tau. $$

Denote
$$\mathcal{D}_k^\tau=\left\{\bigcup_{j=1}^M I_j:  M\leq 2^{k\tau},\, I_j \text{ is an open interval}, \,|I_j|\approx 2^{-k}, \,j=1,\dots,M\right\}.$$

Let $\beta>0$ be a small number that will be specified later. Denote 
$F_k$ as a subset of $F$ which consists of points $x\in F$ who admits some $D_k\in \mathcal{D}_k^\tau$ such that
$$\mu_E(\{y: |y-x|\in D_k\})\geq 2^{-k\beta}.$$
%Each $F_k$ is measurable since $\mu_E$, $\mu_F$ are Borel measures.

For any $D_k\in\mathcal{D}_k^\tau$, denote $\widetilde{D_k}$ as the $2^{-k}$-neighborhood of $D_k$.
\begin{claim*} For any $D_k\in\mathcal{D}_k^\tau$ and any $x\in\R^2$, we have
	$$\mu_E(\{y: |y-x|\in D_k\}) \lesssim \int_{\widetilde{D_k}} d^x_*(\mu_E^{2^{-k}})(t)\,dt.$$
\end{claim*}
\begin{proof}[Proof of Claim]
Notice the right hand side equals
\begin{equation*}\begin{aligned} \int_{|x-z|\in\widetilde{D_k}} \mu_E^{2^{-k}}(z)\,dz= &\, 2^{2k}\iint_{|x-z|\in\widetilde{D_k}} \phi(2^k(z-y))\,d\mu_E(y)\,dz\\\geq&\, 2^{2k}\iint_{|u|\in\widetilde{D_k},\, |x-y-u|\leq 2^{-k-1}}\,du\,d\mu_E(y).\end{aligned}\end{equation*}

Fix $y$ and integrate $u$ first. Since $B(x-y, 2^{-k-1})\subset \widetilde{D_k}$ if $|x-y|\in D_k$, 
this integral is
$$\gtrsim \int_{|y-x|\in D_k}\,d\mu_E(y)=\mu_E(\{y: |y-x|\in D_k\}), $$
as desired.
\end{proof}

Let $R_0 = 2^{\frac{10k\beta}{c\delta}}$, where $c$ is the constant in \eqref{prop-2.1}, and $0<\beta\ll \delta\ll 1$ will be specified later. With this choice we have $RapDec(R_0)=RapDec(2^k)$.

Split $F$ into $F'$ and $F\backslash F'$ such that \eqref{prop-2.1} holds on $F'$ and $\mu_F(F\backslash F')<R_0^{-c\delta}$. Denote $F_k'=F_k\cap F'$. Then $\mu_F(F_k\backslash F_k')<R_0^{-c\delta}\ll 2^{-k\beta}$.

With the claim above and our definition of $F_k$, it follows that
\begin{equation}\label{Lambda_k}2^{-k\beta} \mu_F(F'_k)\leq  \int_{F'_k} \left(\sup_{D_k\in\mathcal{D}_k^\tau}\int_{\widetilde{D_k}}d^x_*(\mu_E^{2^{-k}})(t)\,dt\right)d\mu_F\end{equation}
which is bounded from above by
$$\int_{F'_k} \left(\sup_{D_k\in\mathcal{D}_k^\tau}\int_{\widetilde{D_k}}|d^x_*(\mu_{E, good}^{2^{-k}})(t)|\,dt\right)d\mu_F + \int_{F'_k} \int|d^x_*(\mu_E^{2^{-k}})(t)-d^x_*(\mu_{E, good}^{2^{-k}})(t)|\,dt\,d\mu_F.$$

By \eqref{prop-2.1} the second term is $\lesssim R_0^{-c\delta}\mu_F(F'_k)\ll 2^{-k\beta}\mu_F(F'_k)$, negligible. Therefore
\begin{equation}
	\begin{aligned}
		2^{-k\beta}\mu_F(F'_k)\lesssim & \int_{F'_k} \left(\sup_{D_k\in\mathcal{D}_k^\tau}\int_{\widetilde{D_k}}d^x_*(\mu_{E, good}^{2^{-k}})(t)\,dt\right)d\mu_F\\\leq & \int_{F'_k} \sup_{D_k\in\mathcal{D}_k^\tau} \left(|\widetilde{D_k}|\cdot \int_{\widetilde{D_k}}|d^x_*(\mu_{E, good}^{2^{-k}})(t)|^2 \,dt\right)^\frac{1}{2} d\mu_F\\\leq &\left(\sup_{D_k\in\mathcal{D}_k^\tau} |\widetilde{D_k}|^\frac{1}{2}\right)\cdot\int_{F'_k}  \left(\int|d^x_*(\mu_{E, good}^{2^{-k}})(t)|^2 \,dt\right)^\frac{1}{2} d\mu_F,
	\end{aligned}
\end{equation}
where the second line follows from Cauchy-Schwarz, and the last line follows simply by dropping the integral domain $\widetilde{D_k}$.

Since every $\widetilde{D_k}$ can be covered by $\lesssim 2^{k\tau}$ intervals of length $\approx 2^{-k}$, we have
\begin{equation}\label{Cauchy-Schwarz-again}
\begin{aligned}&\ \ \ \left(2^{-k\beta}\mu_F(F'_k)\right)^2\\&\lesssim \,2^{-k(1-\tau)} \left(\int_{F'_k}  \left(\int|d^x_*(\mu_{E, good}^{2^{-k}})(t)|^2 \,dt\right)^\frac{1}{2} d\mu_F\right)^2\\&\leq\,2^{-k(1-\tau)} \mu_F(F'_k)\iint|d^x_*(\mu_{E, good}^{2^{-k}})(t)|^2 \,dt\, d\mu_F\\&\lesssim_{\delta, \beta} 2^{-k(1-\tau)} \mu_F(F'_k)\, 2^{O(1)k\beta/\delta}\!\int |\widehat{\mu^{2^{-k}}_E}(\xi)|^2|\xi|^{-\frac{s_F+1}{3}+O(\delta)}d\xi + RapDec(2^k)\mu_F(F'_k),
\end{aligned}
\end{equation}
where the third line follows from Cauchy-Schwarz and the last line follows from \eqref{prop-2.2}. Here we need $1<\beta\ll \delta\ll 1$ for $R_0\ll 2^k$ to apply \eqref{prop-2.2}.

Now we can solve for $\mu_F(F'_k)$ to obtain
\begin{equation*}\begin{aligned}\mu_F(F'_k)&\lesssim_{\delta, \beta} 2^{-k(1-\tau-O(\beta/\delta)-2\beta)}\int |\widehat{\mu_E}(\xi)|^2\,|\hat{\phi}(2^{-k}\xi)|^2\,|\xi|^{-\frac{s_F+1}{3}+O(\delta)}\,d\xi+RapDec(2^k)\\&\lesssim_{\delta, \beta} 2^{-k(1-\tau-O(\beta/\delta)-2\beta)}\int_{|\xi|\leq 2^{k/(1-\delta)}} |\widehat{\mu_E}(\xi)|^2\,|\xi|^{-\frac{s_F+1}{3}+O(\delta)}\,d\xi+RapDec(2^k).
\end{aligned}
\end{equation*}

Since $\tau<1$, we may choose $0<\beta\ll\delta\ll 1$ such that $1-\tau-O(\beta/\delta)-3\beta>0$. Then $$2^{-k(1-\tau-O(\beta/\delta)-2\beta)}\lesssim 2^{-k\beta}\cdot |\xi|^{-(1-\delta)(1-\tau-O(\beta/\delta)-3\beta)}=2^{-k\beta}\cdot |\xi|^{-1+\tau+O(\beta/\delta+\delta+\beta)}$$ in the domain $|\xi|\leq 2^{k/(1-\delta)}$, thus
$$\mu_F(F'_k)\lesssim_{\delta, \beta} 2^{-k\beta} \int |\widehat{\mu_E}(\xi)|^2\,|\xi|^{-\frac{s_F+1}{3}-1+\tau+O(\beta/\delta+\delta+\beta)}\,d\xi+RapDec(2^k).$$

Since $s_F>2+3\tau-3 s_E$, we may choose $0<\beta\ll\delta\ll 1$ such that $$-\frac{s_F+1}{3}-1+\tau+O(\beta/\delta+\delta+\beta)<-2+s_E,$$
which guarantees the energy integral
$$\int |\widehat{\mu_E}(\xi)|^2\,|\xi|^{-\frac{s_F+1}{3}-1+O(\beta/\delta+\delta+\beta)}\,d\xi$$
to be finite and therefore $\mu_F(F'_k)\lesssim_{\delta, \beta} 2^{-k\beta}$. 

Above all,
$$\sum_k\mu_F(F_k)=\sum_k\mu_F(F'_k)+\mu_F(F_k\backslash F'_k)\lesssim \sum_k2^{-k\beta}<\infty.$$
By the Borel-Cantelli Lemma, the condition in Lemma \ref{strategy} is satisfied for $\mu_F$-a.e. $x\in F$. Hence by Lemma \ref{strategy},
$$\dH(\Delta_x(E))\geq \tau$$
for $\mu_F$-a.e. $x\in F$, which completes the proof.

\iffalse

\bibliography{/Users/MacPro/Dropbox/Academic/paper/mybibtex.bib}

\fi
\bibliographystyle{abbrv}
\bibliography{/Users/MacPro/Dropbox/Academic/paper/mybibtex.bib}

\end{document}